\documentclass{amsart}
\usepackage{graphicx}
\usepackage{amssymb}
\usepackage{amsfonts}
\usepackage{hyperref}
\usepackage{mathrsfs}
\swapnumbers
\newtheorem{theorem}{Theorem}[section]
\newtheorem{lemma}[theorem]{Lemma}
\newtheorem{corollary}[theorem]{Corollary}
\newtheorem{proposition}[theorem]{Proposition}
\theoremstyle{definition}

\newtheorem{assumption}[theorem]{Assumption}
\newtheorem{remark}[theorem]{Remark}

\numberwithin{equation}{section}
\theoremstyle{plain}

\numberwithin{equation}{section} 
\numberwithin{figure}{section} 
\theoremstyle{plain}
\theoremstyle{plain}
\theoremstyle{remark}
\newtheorem*{acknowledgement*}{Acknowledgement}
\theoremstyle{example}

\newcommand{\cG}{{\mathcal G}}

\newcommand{\cL}{{\mathcal L}}

\newcommand{\te}{{\theta}}

\newcommand{\Om}{{\Omega}}
\newcommand{\om}{{\omega}}

\newcommand{\sig}{{\sigma}}

\newcommand{\eps}{{\epsilon}}

\newcommand{\bbE}{{\mathbb E}}

\newcommand{\bbN}{{\mathbb N}}
\newcommand{\bbP}{{\mathbb P}}
\newcommand{\bbR}{{\mathbb R}}

\newcommand{\bbZ}{{\mathbb Z}}

\begin{document}
\title[Local limit theorem]{Local limit theorem and Edgeworth expansions for inhomogeneous random walks on $GL(d,\bbR)$}   
 \vskip 0.1cm
 \author{Yeor Hafouta}
\address{
Department of Mathematics, Ben-Gurion University}
\email{yeor.hafouta@mail.huji.ac.il}%

\thanks{ }
\dedicatory{  }
 \date{\today}

\maketitle
\markboth{Y. Hafouta}{Local CLT} 
\renewcommand{\theequation}{\arabic{section}.\arabic{equation}}
\pagenumbering{arabic}

\begin{abstract}
We prove a non-lattice local central limit theorem  and Edgeworth expansions for the logarithm of the norms of products of invertible independent random matrices. Our conditions include a contraction assumption, an assumption that supports of the matrices are ``large enough" and their distributions are sufficiently regular. As a byproduct of the proof we are also able to provide a different proof to the optimal rates in the CLT proved in \cite{MatBE}. Like in \cite{MatBE} we provide several sufficient conditions for contraction.
\end{abstract}

\section{Introduction}
The classical central limit theorem (CLT) states that if $(X_j)$ is an iid zero mean sequence of random variables in $L^2$ and $\sig^2=\bbE[X_1^2]>0$ then $(\sig\sqrt n)^{-1}S_n$ converges in distribution to the standard normal law, where $S_n=\sum_{j=1}^nX_j$. The CLT concerns the asymptotic behavior of probabilities of the form $\mathbb P(S_n\in I_n)$ for intervals $I_n$ whose size if of order $\sqrt n$. In this paper we are interested in the local CLT which concerns the asymptotic behavior of probabilities of the form $\mathbb P(S_n\in I)$ with intervals $I$ with unit length. The local CLT has origins in the famous De Moivere-Laplace theorem, and it actually predates the CLT.

In the past decades the local CLT has been  studied extensively for partial sums generated by sufficiently fast mixing homogeneous Markov chains and autonomous chaotic dynamical systems, see \cite{HH}. In this paper we are interested in the local limit theorem for random variables of the form 
$$
S_n(x)=\ln\|A_n\cdots A_2\cdot A_1x\|
$$
where $(A_j)$ is a sequence of independent invertible matrices and $x$ is a unit vector.
Limit theorems for products of iid random matrices $A_j$ have been studied extensively in the past. The CLT for positive matrices was  obtained in \cite{FK61} (even for mixing stationary matrices), see also \cite{Herve}. Since then, there have been many works on limit theorems for products of iid invertible random matrices and processes with values in other groups,  including Berry-Esseen theorems (optimal CLT rates) and local limit theorems  (see \cite{BB,BeQu,BougLac, BRLLT05, BRLLT23, PelMat1, PelMat2, GramaBE, MatAIHP, Guiv2015, Hough2019} and references therein). Recently, also Edgeworth expansions have been obtained (see  \cite{FP, EdgeMat1, EdgeMat2}). 

A natural question concerns the limiting behavior of $S_n(x)$ defined above for independent but not necessarily identically distributed random matrices $A_j$.
The non-stationary additive setting (i.e. partial sums)   has been studied extensively in recent years (see, for instance, \cite{CR07, LD1, DolgHaf PTRF 1, DolgHaf PTRF 2, DolgHaf LLT, DS,Nonlin,HafSPA, NonUBE, NonUEdge, NewBE, HNTV,MPP LLT1, MPP LLT2,NSV12, NTV ETDS 18, CLT3} and references therein and  also \cite{Dob,Pelig,SeVa}). Despite that 
very little is known about the asymptotic behavior of products of non-stationary matrices. 
The main problem is that all existing techniques heavily rely on tools from the theory of stationary processes (see \cite{Guiv2015,HH}). Indeed, (see \cite{Guiv2015}) a typical way is to view the product  acting on a unit vector as a Markov process on the projective space. Another approach is  based on ideas in \cite{KPZ}, that is to view the entire problem as an additive sum over a stationary Bernoulli shift (see \cite{LimMat, PelMat1}). 

Non iid matrices  were addressed for the first time in the recent papers
\cite{Goldsh, GordKlep1, GordKlep2, GordKlep3, MatBE}. In \cite{Goldsh} sufficient conditions for Markov dependent non-stationary matrices $(A_j)$ were provided to ensure that 
$$
\ln\|A_n\cdots A_{2}\cdot A_1\|
$$
grows linearly fast. In 
\cite{GordKlep1} and \cite{GordKlep2} analogous results to the law of large numbers were obtained. 
However, in general exponential growth rates are not expected. For instance, the norms of the matrices might all be smaller than some $c<1$. Yet, the logarithm of the product might exhibit a non-trivial asymptotic behavior. 

As for the CLT, the first result we are aware of   is \cite{GordKlep3}. In that paper, the CLT was obtained for random matrices with values in $\text{SL}(2,\bbR)$ under the same conditions in \cite{GordKlep1}, which ensure that the variance of $\ln\|A_n\cdots A_{2}\cdot A_1\|
$ grows linear fast. Comparing this with the CLT for independent random variables $X_j$ the linear growth is a strong conclusion, since the variance can grow arbitrarily slow. 
Recently, in \cite{MatBE} we proved Berry-Esseen theorems when the action of the matrices is logarithmically contracting on average, a condition which we were able to verify under certain circumstances.

In this paper we  prove local limit theorems and first order correction terms in the CLT (aka Edgeworth expansions)  for invertible logarithmically contracting matrices without restrictions on the growth rate of the variance and without assumptions that lead to exponential growth of the norms. We will also discuss high order Edgeworth expansions under additional conditions. In addition to the contraction assumption, we also assume some regularity properties on the distributions of each $A_j$ and that their supports are large enough. Both conditions are relatively moderate. As a byproduct of the proof we provide a different proof to all of the results in \cite{MatBE}.

Our approach to proving the local CLT and the first order expansions is based on estimating the rate of decay of the characteristic functions of $S_n(x)$. Using the logarithmic contraction and ideas in \cite{DolgHaf PTRF 2} we prove a complex sequential Perron-Frobenius theorem for the underlying Markov operators that appear in the formula of the characteristic functions. Using this, sequential exponential convergence of the Markov operators on the flag space (which we prove), and ideas in \cite{DolgHaf PTRF 1, DolgHaf PTRF 2} we prove the desired control over the characteristic function around the origin. This is what allows us to recover all the results in \cite{MatBE}. 

In order to estimate the  characteristic function away from the origin, we show that the characteristic function of $S_n(x)$ decays exponentially fast in $n$ (and not only in $\text{Var}(S_n(x))$), uniformly on compact subsets of $\mathbb R\setminus\{0\}$. This is much stronger than what is needed to prove the local CLT, and is interesting by its own. In the proof of these exponential bounds we use ideas in \cite{DolgHaf LLT}. We first prove a Lasota-Yorke inequality and then relate the size of the norms of the complex operators controlling the characteristic functions to their values at single points. The reasoning behind this part is similar to  \cite{DolgHaf LLT}, but the arguments differ and rely on certain regularity properties of the distributions of the matrices. 

The second part is to show that the operator norms are uniformly smaller than $1$ after a few iterates. As opposed to \cite{DolgHaf LLT}, this part does not rely on showing that  the underlying norm-cocycle $\sigma(A,y)=\ln\left(\|Ay\|/\|y\|\right)$ is  irreducible. The problem is that  for a unit vector $y$, $\sigma(A,y)$ does not depend only on the direction of $Ay$ (but exactly the opposite). Thus, the perturbation of the Markov operators are not taken with respect to functions on the projective space. To overcome that problem we show directly that the norms of the complex operators are not too close to $1$ when the distribution of $A_j$ is sufficiently regular and the supports of $A_j$'s are large enough. The main idea in the proof is to show that when the norm is close to $1$ then after enough steps the two point action generated by taking an independent copy of $(A_j)$ is close in an appropriate sense to a projective function. Then we use that the  function  $\sigma(A,y)$ is not projective in $Ay$.

Our approach to proving high order Edgeworth expansions also involves decay rates of  the characteristic functions, but on long intervals. Using connectivity assumption on the supports of the matrices $A_j$ we define adapted norms to each argument $t$ of the characteristic functions and then show that we get enough contraction after some number $k(t)$ iterates. Then we assume that the supports of the products $A_{j+n-1}\cdots A_{j}$ has sufficient amount of growth, in an appropriate sense, that will ensure that we can control $k(t)$ and get the desired decay of the   characteristic functions on long, time dependent, intervals.

\section{Preliminaries and main results}
Let $(A_j)_{j\geq 1}$ be a sequence of independent  invertible random matrices of sizes $d\times d, d\geq 2$ (which is not necessarily identically distributed). Denote by $d(\cdot,\cdot)$ the standard metric on the real projective space $\mathbb P^{d-1}(\mathbb R)$. 

\begin{assumption}[Uniform boundedness]\label{Ass0}
 The random variables $N(A_j)=\max(\|A_j\|, \|A_j^{-1}\|)$ are uniformly bounded.   
\end{assumption}
   
\begin{assumption}[Average logarithmic contraction]\label{Ass1}
There exist $n_0\in\bbN$ and $\delta>0$ such that 
$$
\sup_j\sup_{\bar x\not=\bar y}\bbE\left[\ln\left(\frac{d(A_{j,n_0}\bar x,A_{j,n_0}\bar y)}{d(\bar x,\bar y)}\right)\right]\leq -\delta
$$
where $A_{j,n}=A_{j+n-1}\cdots A_j$. Here $\bar x,\bar y\in\mathbb P^{d-1}(\mathbb R)$.
\end{assumption}



Next, let us fix a unit vector $x$ and let 
$$
S_n(x)=\ln\|A_n\cdots A_2\cdot A_1x\|.
$$ 
In this paper we are interested in  local limit theorems and first order correction terms in the CLT for $S_n(x)$, but we begin with the growth of variance of $S_n(x)$.
Denote 
$$
\sigma_n(x)=\sqrt{\text{Var}(S_n(x))}.
$$
Denote also $\sigma(A,y)=\ln\left(\frac{\|Ay\|}{\|y\|}\right)$.
Let the Markov chain $(Y_j)$ given by  $Y_j=(A_j,A_{j-1}\cdots A_1x)$. 
\begin{theorem}[Growth of variance]\label{VAR}
Under Assumptions  \ref{Ass0} and  \ref{Ass1} we have $\lim_{n\to\infty}\sigma_n(x)=\infty$ if and only if there are uniformly bounded functions $R_j$ and $u_j$
such that $(u_j(Y_{j-1},Y_j))$ is a martingale with $\sum_j\text{Var}(u_j(Y_{j-1},Y_j))<\infty$ and
$$
\sigma(Y_j)=u_j(Y_{j-1},Y_j)+R_{j}(Y_j)-R_{j-1}(Y_{j-1}).
$$
\end{theorem}

To state the next results we need additional assumptions.
\begin{assumption}[Regularity of the distributions]\label{Ass2}
For every $\delta>0$, 
$$
\inf_{j\in\mathbb N}\inf_{a\in\text{supp}(A_{j})}\mathbb P(\|A_{j}-a\|\leq\delta)>0
$$
and there exists $k=k_\delta$ such that
$$
\inf_{j\in\mathbb N}\inf_{x,y\in\mathbb P^{d-1}(\mathbb R)}\mathbb P(d(A_{j,k}x,y)\leq\delta)>0.
$$
\end{assumption}
For finite state random matrices the first part of the  assumption holds when $\inf_{k\in\mathbb N}\inf_{a\in\text{supp}(A_k)}\mathbb P(A_k=a)>0$, and it also holds when each $A_k$ has a uniformly bounded away from $0$ density on a sufficiently nice compact subset of the invertible matrices. Of course, we can also have a mixed behavior where some of the matrices take finitely many values and the others have densities, and we can require that the above properties will happen only after a few iterates.

The second part of the assumption means that the action of $(A_{j,k})_k$ is probabilistically  covering/transitive in a  weak sense (as $k$ may depend on $\delta$), and it holds when the supports grow large enough as the number of matrices we multiply (possibly) increases. For instance, it holds if for some projectively transitive subsets of invertible matrices $\Gamma_{j,\delta}$, for all $\delta>0$ we have 
$$
\inf_j\inf_{x\in\mathbb P^{d-1}(\mathbb R)}\inf_{\gamma\in\Gamma_{j,\delta}}\mathbb P(d(A_{j,k_\delta}x,\gamma x)\leq\delta)>0.
$$
Note that when taking $k_\delta=1$ and $\Gamma_{j,\delta}=\Gamma_j$ independent of  $\delta$ then the above condition holds under the first condition of Assumption \ref{Ass2}, assuming that $\Gamma_j\subset\text{supp}(A_j)$.

\begin{assumption}[Size of the supports]\label{Ass3}
One of the following conditions hold:
\vskip0.1cm
(i) The supports of all $A_j$'s are connected and there is $\ell_0\in\mathbb N$ such that the sets 
$$
\Delta_{j,\ell_0}=\left\{|\ln|c|-\ln|d||:\exists x,y\not=0, \|x\|=1 \,\,\exists \alpha,\beta\in\text{supp}(A_{j,\ell_0}),\,\alpha x=cy, \beta x=dy\right\}
$$
satisfy
$$
\inf_j\sup (\Delta_{j,\ell_0})>0.
$$

\vskip0.2cm
(ii) For every $r>1$ there exists $\ell_0\in\mathbb N$ such that the sets
$$
\Gamma_{j,\ell_0}=\left\{\ln|c|-\ln|d|:\exists x,y\not=0, \|x\|=1 \,\,\exists \alpha,\beta\in\text{supp}(A_{j,\ell_0}),\,\alpha x=cy, \beta x=dy\right\}
$$
satisfy
$$
\inf_{j}\inf_{r^{-1}\leq h\leq r}\text{dist}(\Gamma_{j,\ell_0},h\mathbb Z)>0.
$$
\end{assumption}
Both conditions mean that the supports are large enough. Indeed,  if $\text{supp}(A_{j,\ell_0})$ includes a ball of radius $\epsilon$ for some $\epsilon>0$ then we can always take  $\beta=\alpha(1+\epsilon' I)$ for a given $\epsilon'>0$ small enough, $c=1$, $y=\alpha x$ and $d=1+\epsilon$. Using these choices this we see that the set of possible values of $\ln|c|-\ln|d|$ includes an interval which does not depend on $j$, and so condition (i) holds, and the distance from a lattice is uniformly bounded away from $0$, which ensures the validity of condition (ii).

 Note that condition (ii) resembles the assumptions in \cite[page 78]{HH} in the case of iid matrices (as the condition holds when $\Gamma_{j,\ell_0}$ are dense).

Our first result is:
\begin{theorem}[A non-lattice local CLT]\label{LLT}
Under Assumptions \ref{Ass0}, \ref{Ass1}, \ref{Ass2} and \ref{Ass3} and when $\sigma_n(x)\to\infty$,
    for every continuous function with compact support $G:\mathbb R\to\mathbb R$ or an indicator of a bounded interval we have 
$$
\lim_{n\to\infty}\sup_{u\in\mathbb R}\left|\sqrt{2\pi}\sigma_n(x)\mathbb E[G(S_n(x)-u)]-e^{-\frac{(u-\mathbb E[S_n(x)])^2}{2(\sigma_n(x))^2}}\int G(y)dy\right|=0.
$$
 \end{theorem}

 Our second result is:
\begin{theorem}[First order Edgeworth expansions]\label{Edge}
  Under Assumptions \ref{Ass0}, \ref{Ass1}, \ref{Ass2} and \ref{Ass3} and when $\sigma_n(x)\to\infty$, with $\bar S_n(x)=S_n(x)-\mathbb E[S_n(x)]$,
$$
\sup_{t\in\mathbb R}\left|\mathbb P\left(\bar S_n(x)\leq t\sigma_n(x)\right)-\Phi(t)-\kappa_n(x)(t^3-3t)(\sigma_n(x))^{-3}\varphi(t)\right|
=o((\sigma_n(x))^{-1})
$$
where $\Phi(t)=\frac{1}{\sqrt{2\pi}}\int_{-\infty}^t e^{-\frac12 y^2}dy$, $\varphi(t)=\frac{1}{\sqrt{2\pi}}e^{-t^2/2}$ and $\kappa_n(x)=\mathbb E[(\bar S_n(x))^3]$. Moreover, $\kappa_n(x)=O((\sigma_n(x))^2)$.
\end{theorem}
We note that under Assumptions \ref{Ass0} and \ref{Ass1} we get from \cite{MatBE} the optimal CLT rates 
$$
\sup_{t\in\mathbb R}\left|\mathbb P\left(S_n(x)-\mathbb E[S_n(x)]\leq t\sigma_n(x)\right)-\Phi(t)\right|=O((\sigma_n(x))^{-1}).
$$
In fact, as pointed our in Remark \ref{BErem}, the methods of this paper provide an alternative proof of the above rates.
Under our additional assumptions Theorem \ref{Edge} shows that  we have a first order correction term in the above rates.

 Finally, in the connected case and under some assumptions on the growth rate of $k_\delta$ we are able to prove Edgeworth expansions of all orders.
 \begin{theorem}\label{EdgHigh}
 Suppose that the supports of $A_j$ are connected and that $\lim_{n\to\infty}\sigma_n(x)=\infty$. Let 
 Assumptions \ref{Ass0}, \ref{Ass1}, \ref{Ass2} and \ref{Ass3} be in force. Assume that the number $k_\delta$ from Assumption \ref{Ass2} satisfies $k_\delta=O(\delta^{-q})$ for some sufficiently small $q>0$. Then  there are polynomials $P_{j,n,x}$ whose degrees depend only on $j$ and their coefficients are uniformly bounded in $x$ and $n$ such that for all $r\in\mathbb N$, with 
$\bar S_n(x)=S_n(x)-\mathbb E[S_n(x)]$,
 $$
\sup_{t\in\mathbb R}\left|\mathbb P(\bar S_n(x)\leq t\sigma_n(x))-\Phi(t)-\sum_{j=1}^{r}(\sigma_n(x))^{-j}P_{j,n,x}(t)\varphi(t)\right|=o((\sigma_n(x))^{-r}).
 $$
 \end{theorem}

\subsection{Two approaches for the verification of Assumption \ref{Ass1}}
The following section appeared in \cite{MatBE}, and is included here for the sake of completion.
 
The first approach is a perturbative approach. For two invertible matrices $A$ and $B$ set 
$$
\tilde c(A,B)=(\|A\|+\|B\|)\|A-B\|\left(\|A^{-1}\|^2+\|B\|^2\|A^{-1}\|^2\|B^{-1}\|^2\right).
$$
\begin{proposition}\label{Corr1}
Let $(A_j)$ and $(B_j)$ be two independent sequences of random invertible matrices such that Assumption \ref{Ass1} holds for the sequence $(B_j)$. Suppose we can couple $(B_j)$ and $(A_j)$ such that  $\te:=\sup_j\bbE[\tilde c(A_{j,n_0},B_{j,n_0})]<\infty$. Then there exists a constant $\varepsilon_0=\varepsilon_0(\delta,n_0)$ such that  Assumption \ref{Ass1} holds for the sequence $(A_j)$ with some $\delta_1>0$ instead of $\delta$ and the same $n_0$ if $\te<\varepsilon_0$.

In particular, suppose that
 $C_1=\sup_j\|1+N(B_j)\|_{L^{8n_0}}<\infty$ and $C_2=\sup_j\|1+N(B_j)\|_{L^{8n_0}}<\infty$. Assume also
that we can couple $(A_j)$ and $(B_j)$ such that $\sup_j\|A_j-B_j\|_{L^{8n_0}}\leq \varepsilon$. Then there exists a constant $\varepsilon_0=\varepsilon_0(\delta,n_0,C_1,C_2)$ such that  Assumption \ref{Ass1} holds for the sequence $(A_j)$ with some $\delta_1>0$ instead of $\delta$ and the same $n_0$ if $\varepsilon<\varepsilon_0$.  
\end{proposition}

\begin{remark}
  Assumption \ref{Ass1} holds for iid irreducible proximal matrices (see \cite[Proposition 6.4]{BougLac}). Therefore, we can consider small perturbation of iid matrices. 
\end{remark}

The second approach relies on properties of random singular value decompositions (SVD). Let us fix some $n_0\in\bbN$. Let us consider a random SVD 
$$
A_{j,n_0}=U_{j,n_0}\text{Diag}(\sigma_{1,j,n_0},\sigma_{2,j,n_0},...,\sigma_{d,j,n_0})V_{j,n_0}.
$$
where $\sigma_{1,j,n_0}\geq \sigma_{2,j,n_0}\geq\ldots\geq \sigma_{d,j,n_0}$,
and let $u_{1,j,n_0}$ be the first column of $U_{j,n_0}$. 

\begin{proposition}\label{NewProp}
Suppose that there are constants $C,\alpha,n_0>0$ such that for every $\epsilon\in(0,1)$,
\begin{equation}\label{u condd1}
 \sup_j\sup_{\|x\|=1}\bbP(|<u_{1,j,n_0},x>|\leq \epsilon)\leq C(|\ln(\epsilon)|)^{-1-\alpha}.   
\end{equation}
Assume also that there is a constant $\delta>0$ such that for all $j$, 
$$
\bbE[\ln(\sig_{1,j,n_0}/\sig_{2,j,n_0})]\geq 1+C/\alpha+\delta.
$$
Then Assumption \ref{Ass1} is in force with the above $
n_0$ and $\delta$.
\end{proposition}
When $d=2$ the meaning of condition \eqref{u condd1} is that the distribution of the angle of second rotation $U_{j,n_0}$ when viewed as a point on the unit circle assigns mass of order $O(|\ln(\epsilon)|^{-1-\alpha})$ to arches of length $\epsilon$. There is a  similar but more complicated geometric interpretation when $d>2$.

 \section{Growth of the variance, proof of Theorem \ref{VAR}}
 Define a linear operator $\mathcal L_j$ which maps a function $g:\mathbb P^{d-1}(\mathbb R)\to\mathbb C$ to a function $\mathcal L_{j}g$ on $\mathbb P^{d-1}(\mathbb R)$ given by
 $$
\mathcal L_{j}g(x)=\mathbb E[g(A_j x)].
 $$
 Let 
 $$
\mathcal L_{j}^n=\mathcal L_{j}\circ\cdots\circ \mathcal L_{j+n-1}.
 $$
 Then, 
 $$
\mathcal L_{j}^ng(x)=\mathbb E[g(A_{j+n-1}\cdots A_jx)].
 $$
 Next, we need:
\begin{lemma}\label{LLL}
Under Assumptions \ref{Ass0} and \ref{Ass1} for every $\epsilon>0$ small enough  there are $\zeta\in(0,1)$ and $C>0$ such that for all $j$ and $n$ and all $\bar x,\bar y\in\mathbb P^{d-1}(\mathbb R)$,
$$
\mathbb E[d^\epsilon(A_{j,n}\bar x, A_{j,n}\bar y)]\leq C\zeta^n d^\epsilon(\bar x,\bar y). 
$$
\end{lemma}
\begin{proof}
The beginning the proof is based on ideas in the proof of \cite[Lemma 6]{LimMat}.     Let us define 
$$
F_j(\bar x,\bar y)=\bbE\left[\ln\left(\frac{d(A_j\bar x,A_j\bar y)}{d(\bar x,\bar y)}\right)\right]
$$
and 
$$
\sig_j(A, (\bar x,\bar y))=\ln\left(\frac{d(A\bar x,A\bar y)}{A(\bar x,\bar y)}\right)-F_j(\bar x,\bar y).
$$
Then 
$$
\ln\left(\frac{d(A_{j,k}\bar x, A_{j,k}\bar y)}{d(\bar x,\bar y)}\right)=M_{j,k}+R_{j,k}
$$
where 
$$
R_{j,k}=R_{j,k}(\bar x,\bar y)=\sum_{m=1}^{k}F_{j+m}(A_{j,m}\bar x,A_{j,m}\bar y)
$$
and 
$$
M_{j,k}=\sum_{m=1}^{k}\sigma_{j+m}(A_{j+m},(A_{j,m}\bar x,A_{j,m}\bar y)).
$$
Note that for every fixed $j$ we have that $\sigma_{j+m}(A_{j+m},(A_{j,m}\bar x,A_{j,m}\bar y))$ is uniformly (in $j,m$ and $\bar x$ and $\bar y$) bounded martingale difference (as $N(A_s), s\geq 1$ are uniformly bounded random variables).
Notice that by Assumption \ref{Ass1} for all $j$ and unit vectors $\bar x,\bar y$ we have 
$$
\bbE[R_{j,n_0}(\bar x,\bar y)]\leq -\delta.
$$
Using that, by repeating the arguments in the proof of  \cite[Lemma 6]{LimMat} we see that there exist constants $C,\alpha>0$ and $\eta\in(0,1)$ such that 
$$
\sup_j\bbP(R_{j,k}(\bar x,\bar y)\geq -\alpha k)\leq C\eta^k.
$$
Finally, by applying the Azuma inequality and using the Chernoff bounding method we conclude that there is a constant $c>0$ such that for every $j,k$ and $r>0$ we have 
$$
\bbP(M_{j,k}\geq r k)\leq e^{-cr^2 k}.
$$
By taking $r$ small enough we see that the random variables $X_{j,n,\bar x,\bar y}=d(A_{j,n}\bar x,A_{j,n}\bar y)d^{-1}(\bar x,\bar y)$ satisfy 
$$
\mathbb P(X_{j,n,\bar x,\bar y}\geq e^{-cn})\leq C\theta^n
$$
for some $c>0$ and $\theta\in(0,1)$ (which do not depend on $r$ as long as $r$ is small enough). 

Let us take $\epsilon>0$ and let 
$$
X_{j,n,\bar x,\bar y,\epsilon}=d^{\epsilon}(A_{j,n}\bar x,A_{j,n}\bar y)d^{-\epsilon}(\bar x,\bar y) 
$$
Then
$$
\mathbb P(X_{j,n,\bar x,\bar y,\epsilon}\geq e^{-c\epsilon n})\leq C\theta^n.
$$
Now, notice that with $N=\sup_k\|N(A_k)\|_{L^\infty}<\infty$ we have 
$$
0\leq X_{j,n,\bar x,\bar y,\epsilon}\leq C_\epsilon N^{\epsilon n}
$$
for some constant $C_\epsilon>0$. 
In particular, by further taking $\epsilon$ small enough to ensure that $N^{2\epsilon}<\theta^{-1}$ we see that  
$$
\|X_{j,n,\bar x,\bar y,\epsilon}\|_{L^2}\leq C_\epsilon'
$$
for some constant $C_\epsilon'$ which does not depend on $j,n,\bar x$ and $\bar y$. 
Now the result follows since for every nonnegative random variable $X_n$  and all $d>0$ we have 
$$
\mathbb E[X_n]\leq\mathbb E[X_n\mathbb I(X_n\geq e^{-dn})]+e^{-dn}\leq \|X_n\|_{L^2}\sqrt{\mathbb P(X_n\geq e^{-dn})}+e^{-dn}
$$
where the second inequality uses the Cauchy-Schwartz inequality.
\end{proof}

Next, let us fix some small $\epsilon>0$.
Denote by $v_\epsilon(g)$ the H\"older constant of a function $g:\mathbb P^{d-1}(\mathbb R)\to\mathbb C$  corresponding to the exponent $\epsilon$. Denote 
$$
\|g\|_\epsilon=\sup_{x\in\mathbb P^{d-1}(\mathbb R)}|g(x)|+v_\epsilon(g).
$$
Then,  from Lemma \ref{LLL} and the Fubini theorem we get that there is $\gamma\in(0,1)$ and $C>0$ such that for all $j$ and $m>n$ and all $x,y\in\mathbb P^{d-1}(\mathbb R)$,
 $$
|\mathcal L_{j-n}^ng(x)-\mathcal L_{j-m}^n g(y)|\leq C\|g\|_{\epsilon}\gamma^n d^{\epsilon}(x,y).
 $$
  Thus, for fixed $j$, $g$ and $x$ the sequence $(\mathcal L_{j-n}^ng(x))_n$ is Cauchy, and hence it converges to a limit $\nu_j(g)$, which forms is a probability measure $\nu_j$ on $\mathbb P^{d-1}(\mathbb R)$, and 
 \begin{equation}\label{RPF0}
\|\mathcal L_{j-n}^ng-\nu_j(g)\|_\epsilon\leq C\|g\|_{\epsilon}\gamma^n. 
 \end{equation}
 Clearly, $(\mathcal L_j)^*\nu_{j}=\nu_{j+1}$.

Next, define $\sigma(A,y)=\ln\left(\frac{\|Ay\|}{\|y\|}\right)$ for $y\in\bbR^d\setminus\{0\}$. Then 
\begin{equation}\label{N g1}
\sup_{y\not=0}|\sigma(A,y)|\leq |\ln(N(A))|   
\end{equation}
and
$$
S_n(x)=\sum_{j=1}^n\sigma(A_j,A_{j-1}\cdots A_1x).
$$
Note that $\sigma(A,y)=\sigma(A,v)$ if $y$ and $v$ are collinear. Thus we can view the action  above as an action on the projective space. Let $Y_0=(Id,x)$ and
set $Y_j=(A_j,A_{j-1}\cdots A_1 x)$ (where the action in the second coordinate is on the projective space). Then $(Y_j)$ is an inhomogeneous Markov chain with the state space $\text{GL}(d,\mathbb R)\times\mathbb P^{d-1}(\mathbb R)$ at each step. Moreover,
$$
\mathbb E[G(Y_{j+n})|Y_{j}=(y,s)]=\mathbb E[G(A_{j+n},A_{j+n-1}\cdots A_{j+1}ys)]
$$
$$
=\int\left(\mathcal L_{j+1}^{n-1}G(Z,\cdot)\right)(ys)d\mu_{j+n}(Z) 
$$
where $\mu_{j+n}$ is the law of $A_{j+n}$. Denote by $Q_j$ the $j$-th Markov operator of the chain $Y_j$. Let $\kappa_j=\mu_j\times \nu_j$ and $Q_{j,n}=Q_{j}\circ\cdots Q_{j+n-1}$.
We thus conclude that 
\begin{equation}\label{Q}
 \|Q_{j,n}G-\kappa_j(G)\|_{\infty}\leq C\gamma^n\sup_{Z\in\text{supp}(A_{j_n})}\|G(Z,\cdot)\|_\epsilon.   
\end{equation}
Now, we have
$$
S_n(x)=\sum_{j=1}^n\sigma(Y_j).
$$
Let
$$
R_{j}(Y_j)=\sum_{k>j}\mathbb E[\sigma(Y_k)-\mathbb E[\sigma(Y_k)]|Y_j].
$$
Using that $N(A_j)$ are uniformly bounded and that by \cite[Lemma 12.2]{BeQu},
\begin{equation}\label{N g2}
|\sigma(A,y)-\sigma(A,x)|\leq CN(A)d(x,y).     
\end{equation}
We thus see that the above series converges in the supremum norm on functions on $\text{GL}(d,\mathbb R)\times\mathbb P^{d-1}(\mathbb R)$. Moreover, the corresponding supremum norms are uniformly bounded.
Define $u_j$ according to 
\begin{equation}\label{MartCob}
\sigma(Y_j)-\mathbb E[\sigma(Y_j)]=u_j(Y_{j-1},Y_j)+R_j(Y_j)-R_{j-1}(Y_{j-1}).
\end{equation}
Then $u_j$ is a martingale difference, and the proof of Theorem \ref{VAR} is complete.
 \qed
\begin{remark}
Define a function $\rho$ on $\text{GL}(d,\mathbb R)\times \mathbb P^{d-1}(\mathbb R)$ by 
$$
\rho ((A,x),(B,y))=d(Ax,By).
$$
Then the proof shows that 
$$
|Q_{j,n}G(y,s)-Q_{j,n}G(y',s')|\leq C\delta^n\sup_Z\|G(Z,\cdot)\|_{\epsilon}\rho^\eps((y,s),(y',s')).
$$
Thus the functions $R_j$ also satisfy 
$$
|R_j(y,s)-R_j(y',s')|\leq c_0 \rho^\eps((y,s),(y',s'))
$$
for some constant $c_0>0$. In particular, on the support of $Y_j$, 
\begin{equation}\label{Rrr}
|R_j(y,s)-R_j(y,s')|\leq c_1 d^\epsilon(s,s')
\end{equation}
for some constant $c_1>0$.
\end{remark}

\section{Upper bounds on the characteristic functions: proof of Theorems \ref{Edge} and \ref{LLT}}
Arguing like in \cite[Ch.2]{HK} and using \cite[Proposition 25]{DolgHaf PTRF 1},  to prove both theorems it is enough to prove the following: 
\vskip0.1cm
(i) the exist $r_0,C,c>0$ such that for all $t\in[-r_0,r_0]$ we have 
\begin{equation}\label{small t}
|\mathbb E[e^{it S_n(x)}]|\leq Ce^{-ct^2\sigma_n^2(x)}.   
\end{equation}
\vskip0.1cm
(ii) for every $0<r_0<T$ we have 
\begin{equation}\label{large t}
\lim_{n\to\infty}\sigma_n(x)\int_{r_0\leq |t|\leq T}|\mathbb E[e^{it S_n(x)}]|dt=0.
\end{equation}
 \vskip0.1cm
 (iii) with $\Lambda_{n,x}(t)=\ln\mathbb E[e^{itS_n(x)}]$ we have 
\begin{equation}\label{cum}
 \sup_{t\in[-r_0,r_0]}|\Lambda_{n,x}''''(t)|=O((\sigma_n(x))^2)   
\end{equation}
\begin{remark}\label{BErem}
\eqref{cum} implies all the results in \cite{MatBE}, which were proven using a different method.  
\end{remark}
\subsection{The characteristic function around the origin}
For $z\in\mathbb C$ let $\mathcal L_{j,z}$ be the operator mapping a function $g:\bbP^{d-1}(\bbR)\to\mathbb C$ to a function $\mathcal L_{j,z}g$ defined by
$$
\mathcal L_{j,z}g(x)=\mathbb E[e^{z\sigma(A_j,x)}g(A_jx)].
$$
Then $\mathcal L_{j,0}=\mathcal L_{j}$. 
Since $N(A_j)$ are uniformly bounded and because of \eqref{N g1} and \eqref{N g2} we see that the operators $\mathcal L_{j,z}$ are continuous linear operators on the space of H\"older functions with exponent $\epsilon$ (for all $\epsilon$ small enough). Moreover, they have uniformly in $j$ bounded norms and the maps $z\to \mathcal L_{j,z}$ are analytic, uniformly in $j$. Therefore, using also \eqref{RPF0}, we can apply \cite[Theorem D.2]{DolgHaf PTRF 2}. This yields that there exists $r_0>0$ such that for all complex numbers $z$ with $|z|\leq r_0$ there are $\lambda_j(z)\in\mathbb C$, H\"older continuous functions $h_{j}^{(z)}$ with exponent $\epsilon$ on $\mathbb P^{d-1}(\mathbb R)$ and continuous linear functionals $\nu_j^{(z)}$ on the space of H\"older functions such that $\lambda_j(0)=1$, $h_j^{(0)}=1$, $\nu_j^{(0)}=\nu_j$, $\nu_j^{(z)}(1)=\nu_j^{(z)}(h_j^{(z)})=1$. Moreover, $\lambda_j(z),h_j^{(z)}$ and $\nu_j^{(z)}$ are uniformly bounded and analytic in $z$. Furthermore, there are constants $C>0$ and $\eta\in(0,1)$ such that 
$$
\left\|\mathcal L_{j,z}^n-\lambda_{j,n}(z)\nu_{j+n}^{(z)}\otimes h_j^{(z)}\right\|_\epsilon\leq C\eta^n 
$$
where $\mathcal L_{j,z}^n=\mathcal L_{j,z}\circ\cdots\circ\mathcal L_{j+n-1,z}$ and $\lambda_{j,n}(z)=\prod_{k=j}^{j+n-1}\lambda_k(z)$.

Next, we need:
\begin{lemma}\label{LL}
There is a constant $C_1>0$ such that for all $j,n$ we have 
$$
\left\|\sum_{k=j}^{j+n-1}(\sigma(Y_k)-\mathbb E[\sigma(Y_k)])\right\|_{L^4}\leq C_1\left(1+\left\|\sum_{k=j}^{j+n-1}(\sigma(Y_k)-\mathbb E[\sigma(Y_k)])\right\|_{L^2}\right).
$$
\end{lemma}
\begin{proof}
Denote $v_\ell(Y_{j-1},Y_{j})=(u_{\ell}(Y_{\ell-1},Y_\ell))^2$ and let 
$\bar v_\ell(Y_{\ell-1},Y_{\ell})=v_\ell(Y_{\ell-1},Y_{\ell})-\mathbb E[v_\ell(Y_{j\ell-1},Y_{\ell})]$. We claim that there is a constant $C_2>0$ such that
\begin{equation}\label{above}
\mathbb E\left[\left(\sum_{k=j}^{j+n-1}\bar v_k(Y_{k-1},Y_{k})\right)^2\right]\leq C_2\sum_{k=j}^{j+n-1}\mathbb E[(u_k(Y_{k-1},Y_{k}))^2]
\end{equation}
for some constant $C_2>0$. Let us assume \eqref{above} and complete the proof of the lemma based on that.

Recall the following version of Burkholder's inequality for martingales (see \cite[Theorem 2.12]{PelBook}). Let $\mathfrak{d}_1,....,\mathfrak{d}_n$ be a martingale difference with respect to a filtration $(\cG_j)_{j=1}^n$ on a probability space. Let $D_n=\mathfrak{d}_1+\mathfrak{d}_2+...+\mathfrak{d}_n$ and
$E_n=\mathfrak{d}_1^2+\mathfrak{d}_2^2+...+\mathfrak{d}_n^2$.
Then, for every $s\geq 2$ there are constants $c_s,C_s>0$ depending only on $s$ such that 
\begin{equation}\label{Burk}
 c_s\|E_n\|_{L^{s/2}}^{1/2}\leq \|D_n\|_{L^s}\leq C_s\|E_n\|_{L^{s/2}}^{1/2}. 
\end{equation}
Now, applying \eqref{Burk} with the  martingale $(u_k(Y_{k-1},Y_k))$, taking $s=4$ and using \eqref{above} we conclude that 
$$
\left\|\sum_{k=j}^{j+n-1}u_k(Y_{k-1},Y_k)\right\|_{L^4}\leq C_4'\left\|\sum_{k=j}^{j+n-1}u_k^2(Y_{k-1},Y_k)\right\|_{L^2}
$$
$$
=C_4'\left\|\sum_{k=j}^{j+n-1}v_k(Y_{k-1},Y_k)\right\|_{L^2}\leq C_4'\left\|\sum_{k=j}^{j+n-1}\bar v_k(Y_{k-1},Y_k)\right\|_{L^2}
$$
$$
+C_4'\left(\sum_{k=j}^{j+n-1}\mathbb E[(u_k(Y_{k-1}, Y_k))^2]\right)^{1/2}
$$
for some constant $C_4'>0$.
Now, the lemma follows by \eqref{MartCob} and using that $R_j$ are uniformly bounded.

It remains to prove \eqref{above}.
First,  by \eqref{MartCob} with $\bar\sigma_j(Y_\ell)=\sigma(Y_\ell)-\mathbb E[\sigma(Y_\ell)]$ we have
\begin{equation}\label{form}
v_\ell(Y_{\ell-1},Y_{\ell})=(\bar\sigma_\ell(Y_\ell))^2+(R_\ell(\ell_j))^2+(R_{\ell-1}(Y_{\ell-1}))^2+2\bar\sigma_\ell(Y_{\ell})R_{\ell-1}(Y_{\ell-1}) 
\end{equation}
$$
-2\bar\sigma_j(Y_{\ell})R_\ell(Y_\ell)-2R_{\ell-1}(Y_{\ell-1})R_\ell(\ell_j).
$$
Next, using  \eqref{N g2}, \eqref{Q} and \eqref{Rrr}  we see that there exist $C>0$ and $\gamma\in(0,1)$ such that for for $k<\ell-1$,
$$
\left\|\mathbb E[(\bar\sigma_\ell(Y_\ell))^2|Y_{k}]-\mathbb[(\bar\sigma_\ell(Y_\ell))^2]\right\|_{L^\infty}\leq C\gamma^{\ell-k},
$$
$$
\left\|\mathbb E[(R_\ell(Y_\ell))^2|Y_{k}]-\mathbb[(R_\ell(Y_\ell))^2]\right\|_{L^\infty}\leq C\gamma^{\ell-k},
$$
and 
$$
\left\|\mathbb E[(R_{\ell-1}(Y_{\ell-1}))^2|Y_{k}]-\mathbb[(R_\ell(Y_{\ell-1}))^2]\right\|_{L^\infty}\leq C\gamma^{\ell-k}.
$$
Clearly the above estimates also hold when $k=\ell-1$ if we take $C$ large enough since the functions $\sigma$ and $R_\ell$ are uniformly bounded.

Now, let us consider a function of the form $H(Y_{\ell})G(Y_{\ell-1})$ such that both $H$ and $G$ are bounded by some constant $C$ and satisfy
$$
|H(A,s)-H(A,s')|\leq Cd^\epsilon(s,s')
$$
and 
$$
|G(A,s)-G(A,s')|\leq Cd^\epsilon(s,s').
$$
Let us write for $k<\ell-1$, 
$$
\mathbb E[G(Y_\ell)H(Y_{\ell-1})|Y_k]=\mathbb E[(Q_{\ell-1}G(Y_{\ell-1}))H(Y_{\ell-1})|Y_k].
$$
Now, note that the function $L(Y_{\ell-1})=(Q_{\ell-1}G(Y_{\ell-1}))H(Y_{\ell-1})$ is bounded by $C^2$ and satisfy
$$
|L(A,s)-L(A,s')|\leq C'd^\epsilon(s,s')
$$
for some constant $C'>0$. We thus conclude from \eqref{Q},
$$
\left\|\mathbb E[G(Y_\ell)H(Y_{\ell-1})|Y_k]-\mathbb E[G(Y_\ell)H(Y_{\ell-1})]\right\|_{L^\infty}\leq C''\gamma^{\ell-k}
$$
for some constant $C''>0$. Notice that the above also holds with $k=\ell-1$ if we take $C''$ large enough, since the functions are bounded.

Using \eqref{form}, the estimates following it, \eqref{N g1}, \eqref{Rrr} and the above estimate with 
$$
H(Y_{\ell})G(Y_{\ell-1})\in \left\{\bar\sigma_j(Y_{\ell})R_{\ell-1}(Y_{\ell-1}),\, \bar\sigma_\ell(Y_{\ell})R_\ell(Y_\ell), \,R_{\ell-1}(Y_{\ell-1})R_\ell(Y_\ell)\right\}    
$$
 we see that there is a constant $C_0>0$ such that for all $k<\ell$,
\begin{equation}\label{aaa}
 \left\|\mathbb E[v_\ell(Y_{\ell-1},Y_{\ell})|Y_k]-\mathbb E[v_\ell(Y_{\ell-1},Y_{\ell})]\right\|_{L^\infty}\leq C_0\gamma^{\ell-k}.   
\end{equation}

Next, denote $\bar v_\ell(Y_{\ell-1},Y_{\ell})=v_\ell(Y_{\ell-1},Y_{\ell})-\mathbb E[v_\ell(Y_{\ell-1},Y_{\ell})]$. Then by \eqref{aaa},
\begin{equation}\label{above1}
\mathbb E\left[\left(\sum_{k=j}^{j+n-1}\bar v_k(Y_{k-1},Y_{k})\right)^2\right]
\end{equation}
$$
\leq\sum_{j\leq k\leq \ell<j+n}\mathbb E\left[|\bar v_k(Y_{k-1},Y_{k})||\mathbb E[\bar v_\ell(Y_{\ell-1},Y_{\ell})|Y_{k}]|\right]    
$$
$$
\leq C_2\sum_{k=j}^{j+n-1}\mathbb E[v_k(Y_{k-1},Y_{k})]=C_2\sum_{k=j}^{j+n-1}\mathbb E[(u_k(Y_{k-1},Y_{k}))^2]
$$
for some constant $C_2>0$. This completes the proof of \eqref{above}, and the proof of Lemma \ref{LL} is complete.
\end{proof}
\begin{remark}
 Using the inductive argument in the proof of \cite[Proposition 3.3]{DolgHaf PTRF 2}   it follows that for all finite $p\geq 2$ there is a constant $D_p$ such that for all $j,n$ we have 
\begin{equation}\label{Mom}
\left\|\sum_{k=j}^{j+n-1}(\sigma(Y_k)-\mathbb E[\sigma(Y_k)])\right\|_{L^p}\leq D_p\left(1+\left\|\sum_{k=j}^{j+n-1}(\sigma(Y_k)-\mathbb E[\sigma(Y_k)])\right\|_{L^2}\right).    
\end{equation}
\end{remark}

Finally, note that
$$
\mathbb E[e^{it S_n(x)}]=\mathcal L_{0,it}^n\textbf{1}(x).
$$
Using this, Lemma \ref{LL} and the above properties of the operators $\mathcal L_{j,z}$, arguing like in \cite{DolgHaf PTRF 1} we conclude that there exist $r_0,C,c>0$ such that \eqref{small t} and \ref{cum} hold for all $t\in[-r_0,r_0]$.

\subsection{The characteristic function away from $0$}
Here we will prove \eqref{large t}. First, note that $\sigma_n(x)=O(\sqrt n)$. Indeed, by the martingale coboudnary decomposition \eqref{MartCob} and since the functions $u_j,R_j$ are uniformly bounded we conclude that 
$$
\sigma_n(x)\leq 2\sup_j\|R_j\|_{L^\infty}+\sup_j\|u_j\|_{L^2}\sqrt n=O(\sqrt n).
$$
Thus, \eqref{large t} will follow from following, much stronger, result.
\begin{proposition}\label{PPP}
 In the circumstances of Theorems  \ref{Edge} and \ref{LLT}, for every $\epsilon$ small enough for every $0<\delta_0<T$ there exist $c,C>0$ such that
 $$
\sup_{\delta_0\leq |t|\leq T}\|\mathcal L_{0,it}^n\|_{\epsilon}\leq Ce^{-cn}.
 $$
\end{proposition}
\subsubsection{Proof of Proposition \ref{PPP}}
Let us begin with the following result.

\begin{lemma}[Lasota-Yorke inequality]\label{ll1}
For all $\epsilon$ small enough 
there are $C_0>0$ and $\theta_1\in(0,1)$ such that for all $t\in\mathbb R$ and $g:\mathbb P^{d-1}(\mathbb R)\to\mathbb C$, 
$$ 
v_\epsilon(\cL_{j,it}^n g)\leq C_0 \left[|t|\|g\|_\infty+\theta_1^n v_\epsilon(g)\right]
$$
where $\|g\|_\infty=\sup_{x\in\bbP^{d-1}(\bbR)}|g(x)|$. Thus for all $T>0$ 
there is $C_1>0$ and  such that for all $t\in[-T,T]$ and $g:\mathbb P^{d-1}(\mathbb R)\to\mathbb C$, 
$$ 
v_\epsilon(\cL_{j,it}^n g)\leq C_1 \left[\|g\|_\infty+\theta_1^n v_\epsilon(g)\right]
$$
\end{lemma}
\begin{proof}
First, notice that for every $\epsilon\in(0,1]$, all invertible matrices $A$ and all $x,y\in\mathbb P^{d-1}(\mathbb R)$,
$$
|e^{it\ln\|Ax\|}-e^{it\ln\|Ay\|}|\leq |t/\epsilon|\|A\|^\epsilon d^{\epsilon}(x,y).
$$
Therefore, 
$$
\left|e^{it\ln\|A_{j,n}x\|}g(A_{j,n}x)-e^{it\ln\|A_{j,n}y\|}g(A_{j,n}y)\right|
$$
$$
\leq v_\epsilon(g)d^\epsilon(A_{j,n}x,A_{j,n}y)+\|g\|_\infty |t/\epsilon|\|A_{j,n}\|^\epsilon d^{\epsilon}(x,y).
$$
Now, by Lemma \ref{LLL} we see that with $N=\sup_j\|N(A_j)\|_{L^\infty}$,
$$
|\mathcal L_{it,j}^ng(x)-\mathcal L_{it,j}^ng(y)|\leq\mathbb E\left[\left|e^{it\ln\|A_{j,n}x\|}g(A_{j,n}x)-e^{it\ln\|A_{j,n}y\|}g(A_{j,n}y)\right|\right]
$$
$$
\leq Cv_\epsilon(g)d^{\epsilon}(x,y)\zeta^n+N^{n\epsilon}|t/\epsilon|\|g\|_\infty d^{\epsilon}(x,y).
$$
Now the result follows by taking $n$ such that $C\zeta^n<1$ and then iterating.
\end{proof}

Henceforth we fix some $T>0$ and let us work with $t\in[-T,T]$. Let $\|h\|_{\epsilon, T}=\max\left(\|h\|_\infty, \frac{v_\epsilon(h)}{2 C_1}\right)$, where $C_1$ comes from Lemma \ref{ll1}.
We shall abbreviate $\|\cdot\|_{\epsilon, T}=\|\cdot\|_*.$ 
We will need the following result.

\begin{lemma}\label{ll2}
(a) For all $j\geq0$ and $t\in\bbR$ we have $ \|\cL_{j, it}^k h \|_\infty\leq \| \cL_{j}^k |h|\|_\infty. $
\vskip0.1cm
(b) $\forall r\in(0,\frac12)$ $\exists k_1=k_1(r)$ and $\beta(r)>0$ such that for  all $j$ we have the following: if $\|h\|_*\leq 1$ and 
$|h(x)|\leq 1-r$  for some $x$ then $\|\cL_{j}^{k_1} h\|_\infty \leq 1-\beta(r). $
\end{lemma}

\begin{proof}
  (a)   We have 
  $$
|\cL_{j,it}^kh(x)|=|\mathbb E[e^{it\ln\|A_{j,k}x\|}h(A_{j,k}x)]|\leq \mathbb E[|h(A_{j,k}x)|]=(\mathcal L_{j}^k|h|)(x). 
  $$
\vskip0.2cm
  (b) Suppose $|h(x)|\leq 1-r$  for some $x$. Let $\delta\in(0,1)$ and $\Gamma\subset\text{GL}(d,\mathbb R)$. Then for all $k$,
  $$
|\mathcal L_{j}^{k}h(y)|\leq 1-\mathbb P(A_{j,k}\in\Gamma)+\int_{\Gamma} |h(Ay)|d\mu_{j,k}(A)
  $$
  where $\mu_{j,k}$ is the law of $A_{j,k}$. Now, using Assumption \ref{Ass2} there are $\eta(\delta)>0$ and $k_\delta\in\mathbb N$ such that $\mathbb P(d(A_{j,k_\delta}y,x)\leq \delta)\geq\eta(\delta)$. Taking $\Gamma=\{A: d(Ay,x)\leq \delta\}$ and $k=k_\delta$ we see that
 $$
|\mathcal L_{j}^{k_\delta}h(y)|\leq 1-\mathbb P(A_{j,k_\delta}\in\Gamma)+\mathbb P(A_{j,k_\delta}\in\Gamma)(1-r)+2C_1\mathbb P(A_{j,k_\delta}\in\Gamma)\delta^{\epsilon}.
 $$
 The result follows by taking $\delta$ such that $2C_1\delta^\epsilon<r/2$  and using that $\mu_{j,k_\delta}(\Gamma)\geq \eta(\delta)$. 
    \end{proof}

We will also need the following two corollaries of the previous two lemmata.

\begin{corollary}
\label{CorNonExpanding}
Let $C_1$ and $\theta_1$ be the numbers from Lemma \ref{ll1}.
Let $k_0=k_0(C_1)$ be the first positive integer $k$ such that $2C_1\te_1^k\leq 1$. Then
\begin{equation}\label{Norm bound}
\sup_{|t|\leq T}\sup_{j\geq0}\sup_{k\geq k_0}\|\cL_{j,it}^k\|_{*}\leq 1
\end{equation}
where $\|\cL_{j,it}^k\|_{*}$ is the operator norm with respect to the norm $\|\cdot\|_*$. 
\end{corollary}

\begin{proof}
The result follows by combining Lemmata \ref{ll1} and \ref{ll2}(a).
\end{proof}

\begin{corollary}\label{MainCor}
Given $r\in(0,\frac12)$ there exists $k_2=k_2(r)\in \bbN$ with the following properties.
If for some $l$, $m\geq k_0=k_0(C_1)$ and $t\in[-T,T]$ we have 
$\|\cL_{l,it}^{k_2+m}\|_{*}>1- \beta(r)$ (where  $\beta(r)$ comes from Lemma \ref{ll2}), then there exists a function $h$ with $\|h\|_*\leq 1$ such that 
${\min_x} |\cL_{l+k_2,it}^{m}h{(x)}|>1-r$. 
\end{corollary}
\begin{proof}
Let $k_2^*(r)$ be the smallest positive  integer $k$ so that $\te_1^{k}\leq \frac12(1-\beta(r))$, where $\beta(r)$ comes from Lemma \ref{ll2}. Then by Lemma \ref{ll1} for every function $H$ such that $\|H\|_*\leq1$ and all $j\in\bbZ$, $s\geq  k_2^*(r)$ and $t\in[-T,T]$ we have 
\begin{equation}\label{Upppp}
 \frac{v_\epsilon(\cL_{j,it}^{s}H)}{2C_1}\leq 1-\beta(r).
\end{equation}
Next,  take $k_2(r)\!\!=\!\!\max(k_2^*(r), k_1(r))$, where $k_1(r)$ comes from Lemma \ref{ll2}(b). 
Suppose that $\|\cL_{j,it}^{k_2(r)+m}\|_{*}\!>\!1\!\!-\!\! \beta(r)$. Then  there is $h$ such that $\|h\|_*\!\leq\!\! 1$ and $\|\cL_{j,it}^{k_2(r)+m}h\|_*
\!>\!1\!\!-\!\! \beta(r)$. Set $H=\cL_{j,it}^m h$. Then
$$
\|\cL_{j,it}^{k_2(r)+m}h\|_*=\|\cL_{j+m,it}^{k_2(r)}H\|_*=\max\left(\|\cL_{j+m,it}^{k_2(r)}H\|_\infty, \frac{v_\epsilon(\cL_{j+m,it}^{k_2(r)}H)}{2C_1}\right)>1-\beta(r).
$$
Now, since $\|h\|_*\leq 1$ and $m\geq k_0$, it follows from \eqref{Norm bound} that $\|H\|_*\leq 1$. Thus, since $k_2(r)\geq k_2^*(r)$ we conclude from \eqref{Upppp} that 
$$
 \frac{v_\epsilon(\cL_{j+m,it}^{k_2(r)}H)}{2C_1}\leq 1-\beta(r).
$$
Hence 
$
\|\cL_{j+m,it}^{k_2(r)}H\|_\infty>1-\beta(r),
$
and so
by Lemma \ref{ll2}(a),
$
\|\cL_{j+m}^{k_2(r)}H\|_\infty>1-\beta(r).
$
Hence, since $k_2(r)\geq k_1(r)$ by (the contrapositive of)  
Lemma \ref{ll2}(b) we have
$$
\min_{x\in X_{j+m}}|H(x)|=\min_{x\in X_{j+m}}\left|\cL_{j,it}^m h(x)\right|
>1-r
$$
and the proof of the corollary is complete.
\end{proof}
Finally, Proposition \ref{PPP} will follow from the following result together with Lemma \ref{CorNonExpanding}.
\begin{lemma}\label{Ll}
For all $0<\delta_0<T$
there exists $L\in\mathbb N$ such that,
$$
\sup_{\delta_0\leq |t|\leq T}\sup_j\|\mathcal L_{j,it}^L\|_*<1.
$$
As a consequence, there are $C,c>0$ that may depend on $\delta_0$ and $T$ such that
$$
\sup_{\delta_0\leq |t|\leq T}\|\mathcal L_{0,it}^n\|_*\leq Ce^{-cn}.
$$
\end{lemma}
\begin{proof}
Fix some $0<\delta_0<T$ and take some $t$ such that $\delta_0\leq |t|\leq T$.  Let us take $\ell\in\mathbb N$ and assume that for some $j$ and $h$ with $\|h\|_*\leq 1$ we have
\begin{equation}\label{uu}
 \min_{x}|\mathcal L_{j,it}^\ell h(x)|>1-r_\ell   
\end{equation}
with $0<r_\ell<\frac12$. We claim that there is a way to choose $r_\ell$ such that if the above holds true for some $t$ and all $\ell$ then Assumption \ref{Ass3} fails. 
Assuming the validity of that statement we see that for some $\ell=L$ and  all $j$ and $h$ with $\|h\|_*\leq 1$ we have 
$$
\min_{x}|\mathcal L_{j,it}^\ell h(x)|\leq 1-r_\ell
$$
and thus by Corollary \ref{MainCor}, 
$$
\|\mathcal L_{j,it}^L\|_*\leq 1-\beta(r_L). 
$$

Now, let us show that \eqref{uu} with an appropriate choice of $r_\ell$ violates Assumption \ref{Ass3}. In the course of the proof we will work with $r_\ell$ as parameters that will be determined later. First, notice that for all $x$,
$$
1-r_\ell<|\mathcal L_{j,it}^\ell h(x)|\leq (\mathcal L_{j}^\ell|h|)(x).
$$
Now, suppose that for some $y$ we have $|h(y)|\leq \frac12$. Let us take some $\delta\in(0,1)$ and let $\ell>k_\delta$ where $k_\delta$ comes from Assumption \ref{Ass2}. Then 
$$
1-r_\ell<\mathcal L_{j}^{\ell}|h|(x)=\int\mathbb E[|h(A_{j+\ell-k_\delta,k_\delta}(Bx))|]d\mu_{j,\ell-k_\delta}(B)
$$
where $\mu_{j,s}$ is the law of $A_{j,s}$. Now, for a fixed $B$ let 
$$
\Gamma_{B,\delta}=\Gamma_{B}=\{A: d(A(Bx),y)\leq\delta\}.
$$
Then by Assumption \ref{Ass2},
$$
\mathbb P(d(A_{j+\ell-k_\delta,k_\delta}\in\Gamma_B))\geq \eta(\delta)
$$
for some $\eta(\delta)>0$. Since $|h(y)|\leq \frac12$ we thus see that
 $$
\mathbb E[|h(A_{j+\ell-k_\delta,k_\delta}(Bx))|]
\leq 1-\mathbb P(A_{j+\ell-k_\delta,k_\delta}\in\Gamma_B)
$$
$$+\mathbb P(A_{j+\ell-k_\delta,k_\delta}\in\Gamma_B)\frac12+\mathbb P(A_{j+\ell-k_\delta,k_\delta}\in\Gamma_B)(2C_1)\delta^{\epsilon}
\leq 1-(\frac12-2C_1\delta^\epsilon)\eta(\delta).
 $$
 Therefore, if we choose $\delta$ small enough to ensure $2C_1\delta^\epsilon<1/2$ and $r_\ell$ such for all $\ell>k_{\delta}$ we have $r_{\ell}<(\frac12-2C_1\delta^\epsilon)\eta(\delta)$ then we get a contradiction. We thus conclude that if $\ell$ is large enough then $\min_y|h(y)|>\frac12$. 

 Let us write $h=re^{i\phi}$ with $r>1/2$. Note that for any probability measure $\nu$ on a probability space $\Om$ and  measurable functions $q:\Om\to \bbR$  and $r:\Om\to [0,\infty)$ 
$$
\left|\int_{\Om} r(\om)e^{iq(\om)}d\nu(\om)\right|^2=\left(\int r(\omega)d\nu(\omega)\right)^2
$$
$$
-2\int_{\Om}\int_{\Om}\sin^2\left(\frac12(q(\om_1)-q(\om_2)\right)r(\om_1)r(\om_1)d\nu(\om_1)d\nu(\om_2).
$$
Taking $\nu=\nu_x$ given by $\nu_x(\Gamma)=\mathcal L_{j}^{\ell}(\mathbb I_\Gamma)(x)$ and using that $\frac12\leq r\leq 1$ we see that 
$$
\int\int\sin^2\left(\frac12t(\ln\|Ax\|-\ln\|Bx\|)+\frac12(\phi(Ax)-\phi(Bx))\right)d\mu_{j,\ell}(A)d\mu_{j,\ell}(B)\leq 2r_\ell 
$$
where $\mu_{j,\ell}$ is the law of $A_{j,\ell}$. 

Let us fix some $\alpha,\beta$ in $\text{supp}(\mu_{j,\ell})$ and let us write $\alpha=\alpha_{j+\ell-1}\cdots\alpha_j$ and $\beta=\beta_{j+\ell-1}\cdots\beta_j$. 
Let $N=\sup_{k}\|N(A_k)\|_{L^\infty}<\infty$ and $\zeta_\ell$ be an arbitrary sequence that decays to $0$ as $\ell\to\infty$.  Set
$$
\Gamma_\ell=\{(A,B)\in\text{supp}(\mu_{j,\ell})^2: \max(\|A_k-\alpha_k\|,\|B_k-\beta_k\|)\leq(1+N)^{-4\ell}\zeta_\ell, j\leq k<j+\ell\}
$$
where $A=A_{j+\ell-1}\cdots A_j$ and $B=B_{j+\ell-1}\cdots B_j$.
Then by Assumption \ref{Ass2} and independence of the matrices there is $\epsilon_\ell>0$ (which depends only on $\ell$) such that 
$$
(d\mu_{j,\ell}\times d\mu_{j,\ell})(\Gamma_\ell)\geq \epsilon_\ell.
$$
In fact, we can take 
$$
\epsilon_\ell\asymp (q((1+N)^{-4\ell}\zeta_\ell))^{2\ell}
$$
where 
$$
q(\delta)=\inf_{j\in\mathbb N}\inf_{a\in\text{supp}(A_{j})}\mathbb P(\|A_{j}-a\|\leq\delta).
$$
Notice that when $(A,B)\in\Gamma$ we have 
$$
\max(\|A-\alpha\|, \|B-\beta\|)\leq (1+N)^{-3\ell}\zeta_\ell.
$$
Using that $|\ln a-\ln b|\leq \frac{|b-a|}{\min(a,b)}, a,b>0$, that  $v_\epsilon(\phi)\leq C$ for some constant $C$
and that on $\Gamma$ we have $\max(d(Ax,\alpha x),d(Bx,\beta x))\leq C'N^{\ell}\max(\|A-\alpha\|, \|B-\beta\|)$, we see that for $(A,B)\in\Gamma$, 
$$
\sin^2\left(\frac12t(\ln\|Ax\|-\ln\|Bx\|)+\frac12(\phi(Ax)-\phi(Bx))\right)
$$
$$
\geq \sin^2\left(\frac12t(\ln\|\alpha x\|-\ln\|\beta x\|)+\frac12(\phi(\alpha x)-\phi(\beta x))\right)-2C''|t|\zeta_\ell^\epsilon
$$
for some constant $C''>0$.
We thus conclude that for all choices of $\alpha,\beta$ and $x$,
$$
\sin^2\left(\frac12t(\ln\|\alpha x\|-\ln\|\beta x\|)+\frac12(\phi(\alpha x)-\phi(\beta x))\right)\leq 2r_\ell \epsilon_\ell^{-1}+2C''|t|\zeta_\ell^\epsilon.
$$

Finally, in the connected case we note that if $\alpha,\beta\in\text{supp}(A_{j,\ell_0})$ satisfy $\alpha x=cy$ and $\beta x=dy$ then for every matrix $\gamma$ we have $\gamma(\alpha x)=c\gamma y$ and $\gamma(\beta x)=d\gamma y$. We thus conclude that for all $\ell>\ell_0$, where $\ell_0$ comes from Assumption \ref{Ass3},
\begin{equation}\label{q ell}
q_\ell:=\inf_j\sup(\Delta_{j,\ell})\geq q_{\ell_0}>0    
\end{equation}
where 
$$
\Delta_{j,\ell}=\left\{|\ln|c|-\ln|d||:\exists x,y\not=0, \|x\|=1 \,\,\exists \alpha,\beta\in\text{supp}(A_{j,\ell}),\,\alpha x=cy, \beta x=dy\right\}.
$$
Next, since $(\alpha,\beta,x)\to t(\ln\|\alpha x\|-\ln\|\beta x\|)+\frac12(\phi(\alpha x)-\phi(\beta x))$ is continuous and vanishes when $\alpha=\beta$ then 
$$
t(\ln\|\alpha x\|-\ln\|\beta x\|)+\phi(\alpha x)-\phi(\beta x)=O(r_\ell \epsilon_\ell^{-1}+|t|\zeta_\ell^\epsilon).
$$
Now, let us take two points $x,y$ and matrices $\alpha$ and $\beta$ such that $\alpha x=cy$ and $\beta x=dy$. Then, since $\phi$ is a function on the projective space $\phi(\alpha x)=\phi(\beta x)$ and so there is a constant $C>0$ such that
$$
\left|t(\ln|c|-\ln|d|)\right|\leq C(r_\ell \epsilon_\ell^{-1}+|t|\zeta_\ell^\epsilon).
$$
Taking $r_\ell$ and $\zeta_\ell$ small enough to ensure that 
$$
C(r_\ell \epsilon_\ell^{-1}/\delta_0+\zeta_\ell^\epsilon)/\delta_0<q_\ell
$$
 with $q_\ell$ like in \eqref{q ell} (recall also that $\delta_0\leq |t|$), we get a contraction to \eqref{q ell}.

When the supports are not necessarily connected we note that, as above,
or every $r>1$ there exists $\ell_0>0$ such that with
$$
\Gamma_{j,\ell}=\left\{\ln|c|-\ln|d|:\exists x,y\not=0, \|x\|=1 \,\,\exists \alpha,\beta\in\text{supp}(A_{j,\ell}),\,\alpha x=cy, \beta x=dy\right\}
$$
we have $\Gamma_{j,\ell_0}\subset \Gamma_{j,\ell}$ for all $\ell\geq \ell_0$. Thus, for all $r>1$,
$$
\inf_{j}\inf_{r^{-1}\leq h\leq r}\text{dist}(\Gamma_{j,\ell},h\mathbb Z)\geq\inf_{j}\inf_{r^{-1}\leq h\leq r}\text{dist}(\Gamma_{j,\ell_0},h\mathbb Z) >0.
$$
Next, from the previous estimate we get that  there is an integer valued  function $Z_{j,\ell}(\alpha x,\beta x)$ such that 
$$
t(\ln\|\alpha x\|-\ln\|\beta x\|)+\phi(\alpha x)-\phi(\beta x)=2\pi Z_{j,\ell}(t,\alpha x,\beta x)+O(2r_\ell \epsilon_\ell^{-1}+\zeta_\ell^\epsilon).
$$
Taking again $\alpha,\beta$ $x,y$ and $c,d$ such that $\alpha x=cy, \beta x=dy$ we see that the set of all possible values $\ln|c|-\ln|d|$ is $O(r_\ell \epsilon_\ell^{-1}+\zeta_\ell^\epsilon)$ close to the lattice $(2\pi/t)\mathbb Z$, which by taking $r_\ell$ and $\zeta_\ell$ small enough (but uniformly in $\delta_0\leq |t|\leq T$) contradicts the second possibility in Assumption \ref{Ass3}.
\end{proof}
\qed

\section{High order Edegworth expansions in the connected case; proof of Theorem \ref{EdgHigh}}
First, using \eqref{Mom}, the properties of the operators $\mathcal L_{j,z}$ and the arguments in \cite{DolgHaf PTRF 1} it follows that for all $m\in\mathbb N$ there is $r_m>0$ such that 
 with $\Lambda_{n,x}(t)=\ln\mathbb E[e^{itS_n(x)}]$ we have 
\begin{equation}\label{cum1}
 \sup_{t\in[-r_m,r_m]}|\Lambda_{n,x}^{(m)}(t)|=O((\sigma_n(x))^2).   
\end{equation}
Combining this with
\cite[Proposition 25]{DolgHaf PTRF 1} and Proposition \ref{PPP}
in order to prove Theorem \ref{EdgHigh} it is enough to show that for all $a,b\geq 1$ and all $r\geq 2$, 
$$
\int_{b\leq |t|\leq an^{(r-1)/2}}\frac{|\mathbb E[e^{itS_n(x)}]|}{|t|}dt=o(n^{-r/2})
$$
where we have taken into account that $\sigma_n(x)=O(\sqrt n)$.

Now, let us take some $b\geq 1$. Let us fix some real $t$ such that $|t|\geq b$ and define a norm by setting
$$
\|h\|_t=\max\left(\|h\|_\infty, \frac{v_\epsilon(h)}{C|t|}\right)
$$
where $C$ is a sufficiently large constant. Then repeating the arguments in the proof of  Proposition \ref{PPP} reveals that in Lemma \ref{ll2}  we can take $k_1(r)=k_1(r,t)=k_{(\frac{r}{4C|t|})^{1/\epsilon}}$ where $k_\delta$ comes from Assumption \ref{Ass2}. Moreover, in the proof of Lemma \ref{Ll} to ensure that $\min_y|h(y)|>1/2$ we can take $\ell>k_{(\frac{1}{8C|t|})^{1/\epsilon}}$. Taking $r_\ell$ and $\zeta_\ell$ small enough in that proof we see that there  are constants $c,r_0>0$ such that with 
$$
k_1(t)=\max\left(k_{(\frac{r_0}{4C|t|})^{1/\epsilon}},k_{(\frac{1}{8C|t|})^{1/\epsilon}}\right)
$$
for all $j$ and all $\ell\geq k_1(t)$ we have 
$$
\|\mathcal L_{j,it}^{k_1(t)}\|_{t}\leq 1-c.
$$
We thus conclude that for some constant $\lambda>0$ we have
$$
\|\mathcal L_{0,it}^{n}\|_{t}\leq e^{-\lambda n/k_1(t)} 
$$
and so
$$
\int_{b\leq |t|\leq an^{(r-1)/2}}\frac{|\mathbb E[e^{itS_n(x)}]|}{|t|}dt\leq \int_{b\leq |t|\leq an^{(r-1)/2}}e^{-\lambda n/k_1(t)}dt. 
$$
Now under the assumptions of Theorem \ref{EdgHigh} we have $k_1(t)=o(n^w)$ when $|t|\leq an^{(r-1)/2}$ for some $w<1$, and thus the above right hand side is of order $o(n^{-r/2})$.


\end{document}